\documentclass[11pt]{amsart}
\usepackage{amsfonts,amssymb,amsthm,eucal,amsmath}
\allowdisplaybreaks
\newtheorem{thm}{Theorem}

\newtheorem{lemma}[thm]{Lemma}

\newcommand{\R}{\mathbb{R}}
\newcommand{\Z}{\mathbb{Z}}
\newcommand{\E}{\mathbb{E}}
\newcommand{\N}{\mathbb{N}}

\newcommand{\s}{\mathbb{S}}
\newcommand{\T}{\mathbb{T}}
\renewcommand{\r}{\mathfrak{R}}

\newcommand{\inprod}[2]{\left\langle #1, #2 \right\rangle}

\renewcommand{\P}{\mathbb{P}}

\newcommand{\var}{\mathrm{Var}}

\newcommand{\dv}{d\overline{\rm vol}}
\DeclareMathOperator{\tr}{Tr\,}

\newcommand{\n}{\mathfrak{N}}
\newcommand{\V}{\mathcal{V}}

\begin{document}

\title[Eigenfunctions of the Laplacian]{On the approximate normality of eigenfunctions of the Laplacian}
\author[E.\ Meckes]{Elizabeth Meckes}
\email{esmeckes@math.cornell.edu}
\begin{abstract}
The main result of this paper is a bound on the distance between the
distribution of an eigenfunction of the Laplacian on a compact
Riemannian manifold and the Gaussian distribution.
If $X$ is a random point on a manifold $M$ and $f$ is an eigenfunction
of the Laplacian with $L^2$-norm one and eigenvalue $-\mu$, 
then $$d_{TV}(f(X),Z)\le\frac{2}{\mu}\E\big|\|\nabla f(X)\|^2-\E\|\nabla f(X)
\|^2\big|.$$  This result is applied to construct specific examples of
spherical harmonics of arbitrary (odd) degree which are close to 
Gaussian in distribution.  A second application is given to random 
linear combinations of eigenfunctions on flat tori.
\end{abstract}

\maketitle


\section{Introduction}
A question which has arisen and been studied often throughout the history
of probability has been to understand the value distributions of certain
naturally occurring functions.  The simplest example of such a question
might be to understand the distribution of the number of times out of $N$
trials that a fair coin lands on heads.  A more geometric example 
(studied by Maxwell and Borel among others) is 
understanding the value distribution of the function $f:\s^{n-1}\to\R$
defined by $f(x)=x_1$.  Value distribution in this context refers
to endowing the sphere $\s^{n-1}$ with a natural probability
measure (normalized surface area), and considering the distribution of 
the random variable which is the push-forward of that natural measure on the 
sphere to $\R$ by the function $f$; i.e., if $X$ is a random point on 
the sphere, the value distribution of $f$ is the distribution of the 
random variable $f(X)$ on $\R$.

A particular class of functions of interest are eigenfunctions of the Laplacian
on Riemannian manifolds.  The main result of this paper is to 
give a sufficient condition for the approximate normality of these functions,
as follows.

\begin{thm}\label{intromain}
Let $M$ be a compact Riemannian manifold 
(without boundary), 
and
$f$ an eigenfunction
for the Laplacian on $M$ with eigenvalue $-\lambda<0$, normalized
so that $\frac{1}{vol(M)}\int_Mf^2=1$.  Let $X$ be a random (i.e., distributed 
according to normalized volume measure) point of $M$.
Then
$$d_{TV}(f(X),Z)\le\frac{2}{\lambda}\E\Big|\|\nabla f(X)\|^2-\E\|\nabla f(X)\|
^2\Big|,$$
where $d_{TV}$ denotes the total variation distance between random 
variables and $Z$ is a standard Gaussian random variable on $\R$.
\end{thm}
That is, the value distribution of $f$ is close to Gaussian if the
expected distance between the random variable $\|\nabla f(X)\|^2$ and
its mean is small, relative to the size of the eigenvalue.  It should 
be pointed out that $\E\|\nabla f(X)\|^2=\lambda$ (this is an easy consequence
of Stokes' theorem), thus the bound on the right-hand side is not always
non-trivial.  However, it will be seen in examples below that there is 
sometimes sufficient cancellation to obtain a normal limit.

\medskip

Eigenfunctions of the Laplacian have arisen as important and 
natural objects in several branches of analysis; in particular, they 
of central interest
in quantum chaos, whose basic problem is the understanding of the quantization
of a classical Hamiltonian system with chaotic dynamics.  
See \cite{sarnak} for a readable introduction.  A key question
in the chaotic case in whether individual eigenfunctions behave like random
waves (see \cite{berry}).  It is noted in \cite{sarnak} that, were this 
the case, it would imply that the distribution of high eigenfunctions 
 would be approximately Gaussian.
In connection with this conjecture,
Hejhal and various coauthors (see, e.g., \cite{hej}) and 
Aurich and Steiner \cite{as} have provided
numerical evidence indicating that on certain hyperbolic manifolds,
the distributions of eigenfunctions become close to Gaussian as
the eigenvalue tends to infinity.  Theorem \ref{intromain} above 
provides a different perspective on the approximate normality of
eigenfunctions.  In the applications given below, asymptotic results
are obtained as the dimension of the manifold in question tends to infinity.
However, there is no {\em a priori} reason that Theorem \ref{intromain} 
could not also be applied in the case of a fixed manifold, as the 
eigenvalue tends to infinity.

The results of this paper can be seen as an abstraction of earlier 
work of the author \cite{morth} and joint work with M. Meckes \cite{MM}.
In those papers, it was shown that the value distributions of linear functions
on the orthogonal group and on certain convex bodies are 
close to Gaussian, using similar methods to those used here.  
The proof of Theorem \ref{intromain} makes clear that eigenfunctions
of the Laplacian are the natural candidates (at least for the techniques
used here) to replace linear
functions on more general Riemannian manifolds.  The main result of 
\cite{morth} on the orthogonal group and the results of \cite{MM} in 
the case of the sphere can be recovered as applications of Theorem 
\ref{intromain}.

\bigskip

The contents of this paper are as follows.  
Notation and conventions are
given in Section \ref{notation} below.  
Section \ref{abstract} gives the proof of the main theorem,
in fact proving a slightly more general version which treats linear 
combinations
of eigenfunctions whose eigenvalues are close together.  Section \ref{sphere}
gives applications of Theorem \ref{intromain} to the distribution of
certain eigenfunctions on the sphere $\s^{n-1},$ and  
section \ref{torus}
contains applications of the generalization of Theorem \ref{intromain} from
section \ref{abstract} to random eigenfunctions on the torus $\T^n$.
The unifying idea in these two cases is that in situations 
in which eigenspaces have high dimension, or the closely related
situation in which there are many low dimensional eigenspaces whose 
eigenvalues are very close together, Gaussian distributions occur when 
taking linear combinations of many orthogonal eigenfunctions.  

In the case of the sphere, we have the following result for the second-order
eigenfunctions.
\begin{thm}\label{quad-intro}
Let $g(x)=\sum_{i,j}a_{ij}x_ix_j,$ where $A=(a_{ij})_{i,j=1}^n$ is a
symmetric matrix with $\tr(A)=0$.  Let $f=Cg$, where $C$ is chosen such that
$\|f\|_2=1$ when $f$ is considered as a function on $S^{n-1}$.  Let $W=f(X),$ 
where $X$ is a random point on $S^{n-1}$.  If $d$ is the 
vector in $\R^n$ whose entries are the eigenvalues of $A$, then 
\begin{equation*}
d_{TV}(W,Z)\le \sqrt{6}\left(\frac{\|d\|_4}{\|d\|_2}\right)^2,\end{equation*}
where $\|d\|_p=\left(\sum_i|d_i|^p\right)^{1/p}.$
\end{thm}
This theorem gives a fairly complete picture of when second-order 
eigenfunctions of the spherical Laplacian have Gaussian value distributions.
The bound on the right-hand side is of the order $\frac{1}{\sqrt{n}}$
in the case that all of the $d_i$ are of roughly the same size; e.g., when
$$d_i=\begin{cases}1&i\le\frac{n}{2}\\-1&\frac{n}{2}<i\le n\end{cases}$$
and $n$ is even.  In the opposite extreme case, e.g. for the function
$f(x)=c(x_1^2-x_2^2)$ for a suitable normalization constant $c$, the
bound is of order 1.  This is what is expected in this case; since it 
is known that $x_1$ and $x_2$ are asymptotically distributed as independent
Gaussians (as the dimension tends to infinity), 
the random variable $f(X)$ is approximately distributed
as $Z_1^2-Z_2^2$ for independent normal random variables $Z_1$ and $Z_2$
and $X$ a random point of a high-dimensional sphere.

In the higher degree case, we have the following theorem.
\begin{thm}\label{degl-intro}
There is a set $\{p_\ell^{(i)}\}_{i=1}^n$ of orthonormal eigenfunctions
of degree $\ell$ on $\s^{n-1}$ and a constant $c$ depending on $\ell$ 
such that 
for $p:\s^{n-1}\to\R$ defined by 
$$\sum_{i=1}^na_ip_\ell^{(i)}(x)$$ 
with $\sum a_i^2=1$, 
$$d_{TV}(p(X),Z)\le c\|{\bf a}\|_4^2.$$
If the vector of coefficients $\{a_i\}$ is chosen to be random and 
uniformly distribtuted on the sphere, then there is another constant $c'$
such that 
$$\E d_{TV}(p(X),Z)\le \frac{c'}{\sqrt{n}}.$$
\end{thm}
The functions $p_\ell^{(i)}$ are given explicity in Section \ref{sphere}.
The second statement of the theorem says that the typical linear 
combination of the $p_\ell^{(i)}$ has an approximately Gaussian value 
distribution.  As in the second-degree case, the explicit form
of the bound in the first statement of the theorem means that for 
specific choices of ${\bf a}$ whose components all have roughly the 
same magnitude; e.g., $a_i=\frac{1}{\sqrt{n}}$ for each $i$, the 
bound is of order $\frac{1}{\sqrt{n}}.$  

\smallskip

In the case of the torus, we have the following result on 
the distributions of random linear combinations of eigenfunctions.
\begin{thm}\label{torus-intro}
Let $B$ be a symmetric, positive definite matrix, and consider the
flat torus $\R^n/\Z^n$ with the metric given by $(x,y)_B=\inprod{Bx}{y}.$
Define a random function $f$ on $(\T^n,B)$ as follows.  Let $\{a_v\}_{v\in\V}$ 
be a random point of the sphere of radius $\sqrt{2}$ in $\R^\V$, where $\V$
is a finite set of vectors of $\R^n$.  Suppose that the set $\V$ is such
that if $v\in\V$, then $Bv\in\Z^n.$  Define $f$ by
$$f(x)=\Re\left(\sum_{v\in\V}a_ve^{2\pi i\inprod{Bv}{x}}\right).$$
Then for $\mu_v=(2\pi\|v\|_B)^2$ and $\mu>0$,
$$\E d_{TV}(f(X),Z)\le\frac{1}{|\mu|}\left[\sqrt{\frac{8(2\pi)^4}
{|\V|(|\V|+2)}\sum_{v,w\in\V}(v,w)_B^2}+
\left(2\sqrt{2}+\sqrt{\pi}\right)\sqrt{\frac{1}{|\V|}
\sum_{v\in\V}(\mu_v-\mu)^2}\right].$$
\end{thm}
Section \ref{torus} contains a discussion of the bounds above and 
concrete examples in which they are small.  The simplest possible
case is that of the usual square-lattice torus, in which the matrix $B=I$.
In that case, consider the set $\V$ to be simply the standard basis
vectors of $\R^n$.  Then the fact that $f$ as constructed above 
has a value distribution which
is close to normal is essentially 
just the classical central limit theorem.  In 
that case, the second term of the bound above is zero (take $\mu=(2\pi)^2$),
and the first term is $\frac{2\sqrt{2}}{\sqrt{n}}.$  

In the case of more general $B$ and $\V$, the first term is a weakening
of the orthogonality in the previous case; it is small as long as $\V$ is
large but the 
elements of $\V$ have directions which are fairly evenly distributed on 
the sphere.  The second term is a weakening of the fact that, in the previous
example, all of the summands in the linear combination were from the same 
eigenspace.  By taking $B$ to be a slight perturbation of $I$, the
eigenspaces may become low- or even one-dimensional.  However,
 the corresponding
eigenvalues may still be clustered very close together, in which case the 
second term of the error above is small.  
\medskip

\subsection{Notation and Conventions}\label{notation}
For random variables $X$ and $Y$, the total variation distance between
$X$ and $Y$ is defined by 
$$d_{TV}(X,Y)=\sup_A\big|\P(X\in A)-\P(Y\in A)\big|,$$
where the supremum is taken over measureable sets $A$.  This is 
equivalent to
$$d_{TV}(X,Y)=\frac{1}{2}\sup_f\big|\E f(X)-\E f(Y)\big|,$$
where the supremum is taken over functions $f$ which are continuous
and bounded by one.

Let $(M,g)$ be an $n$-dimensional Riemannian manifold.  Integration
with respect to the normalized volume measure is denoted $\dv$, thus
$\int_M1\dv=1.$
For coordinates $\left\{\frac{\partial}{\partial 
x_i}\right\}_{i=1}^n$ on $M$, define 
\begin{equation*}
(G(x))_{ij}=g_{ij}(x)=\inprod{\left.\frac{\partial}{\partial x_i}\right|_x}
{\left.\frac{\partial}{\partial x_j}\right|_x},\qquad g(x)=\det(G(x)),
\qquad g^{ij}(x)=(G^{-1}(x))_{ij}.
\end{equation*}
Define the gradient $\nabla f$ of $f:M\to\R$ and the Laplacian
$\Delta_g f$ of $f$ by
\begin{equation*}
\nabla f(x)=\sum_{j,k}\frac{\partial f}{\partial x_j}g^{jk}
\frac{\partial }{\partial x_k},\qquad\qquad \Delta_g f(x)
=\frac{1}{\sqrt{g}}\sum_{j,k}\frac{\partial}{\partial x_j}
\left(\sqrt{g}g^{jk}\frac{\partial f}{\partial x_k}\right).
\end{equation*}

Let 
$\Phi^t(x,v)$ denote the geodesic flow on the tangent bundle $TM$ and 
let $\pi:TM\to M$ be the projection map of a tangent vector onto its base 
point.

\medskip

{\bf Acknowledgements.}  I would like to thank Pierre Albin, 
Persi Diaconis, Geir 
Helleloid, Ze\'ev Rudnick, Laurent
Saloff-Coste, and Steve Zelditch for helpful discussions.  In particular,
I learned about hypergeometric series from Geir Helleloid and about 
Liouville measure from Steve Zelditch.

\section{The main result}\label{abstract}

As mentioned above, the main result of this section is a slight 
generalization of Theorem \ref{intromain} of the previous section.  The
idea is that for the method of proof used here, the functions considered
need only be {\em approximate} eigenfunctions of $\Delta_g$; in particular,
linear combinations of eigenfunctions all of whose eigenvalues cluster about
one value are close enough.

\begin{thm}\label{absmfld}
Let $(M,g)$ be a compact Riemannian manifold without boundary, and 
let $\{f_i\}_{i=1}^m$ be 
a sequence of mutually orthogonal eigenfunctions of the Laplacian 
$\Delta_g$ on $M$ with
eigenvalues $-\mu_i<0$.  Assume that $\int_Mf_i^2\dv=1$ for each $i$.  Define 
the function $f:M\to\R$ by
$$f(x)=\sum_{i=1}^ma_if_i(x),$$
where the $a_i$ are scalars such that $\sum a_i^2=1.$
Let $X$ be a random (i.e., distributed 
according to normalized volume measure) point of $M$.  
Then for $\mu>0$,
$$d_{TV}(f(X),Z)\le\frac{2}{\mu}\left[\E\left|\|\nabla f(X)\|^2-
\E\|\nabla f(X)\|^2\phantom{\sum_i}\hspace{-5.5mm}\right|+
\left(1+\frac{\sqrt{\pi}}{2\sqrt{2}}\right)\sqrt{\sum_ia_i^2(\mu_i-\mu)^2}\right],$$
where $Z$ is a standard Gaussian random variable on $\R$.
\end{thm}

The idea is to choose $\mu=\frac{1}{m}\sum \mu_i,$ and consider 
collections $\{f_i\}$ of eigenfunctions such that the eigenvalues
$\mu_i$ are very close together.  In particular, for manifolds on
which eigenspaces of $\Delta$ have high dimension, one can choose all the
$\mu_i$ to be equal, in which case the second error term drops out.

\bigskip

The proof is an application of the following abstract normal approximation
theorem.  A proof of the theorem was given in \cite{morth} in the case
that $E'=0$; the proof of the theorem below is a trivial modification of the
earlier case.

\begin{thm}\label{abscont}Suppose that $(W,W_\epsilon)$ is a family of 
exchangeable pairs defined on a common probability space with 
$\E W=0$ and $\E W^2=\sigma^2$.  Suppose there are random variables $E=
E(W)$ and $E'=E'(W)$, 
deterministic functions $h$ and $k$ with 
$$\lim_{\epsilon\to0}h(\epsilon)=\lim_{\epsilon\to0}k(\epsilon)=0$$
and functions $\alpha$ and $\beta$ with 
$$\E|\alpha(\sigma^{-1}W)|<\infty,\qquad\E|\beta(\sigma^{-1}W)|<\infty,$$
such that 
\begin{enumerate}
\item $$\frac{1}{\epsilon^2}\E\left[W_\epsilon-W\big|
W\right]=-\lambda W+E'+h(\epsilon)\alpha(W),$$\label{lindiff3}
\item $$\frac{1}{\epsilon^2}\E\left[(W_\epsilon-W)^2\big|
W\right]=2\lambda\sigma^2 +E\sigma^2+k(\epsilon)\beta(W),$$\label{quaddiff3}
\item $$\frac{1}{\epsilon^2}\E\left|W_\epsilon-W\right|^3
=o(1),$$\label{cubediff3}
\end{enumerate}
where $o(1)$ refers to the limit as $\epsilon\to0$ with implied
constants depending on the distribution of $W$.
Then
$$d_{TV}(W,Z)\le \frac{1}{\lambda}\left[\sqrt{\frac{\pi}{2}}
\E\left|E'\right|+\E\left|E\right|\right],$$
where $Z\sim\n(0,\sigma^2).$

\end{thm}

\bigskip

\begin{proof}[Proof of Theorem \ref{absmfld}]
To apply Theorem \ref{abscont} above, start by constructing a family of 
exchangeable pairs of points in $M$ 
parametrized
by $\epsilon$.  Let $X$ be a random point of $M$ and
let $\epsilon>0$ be smaller than the injectivity radius of the exponential
map at $X$.  Since $M$ is compact, 
there is a range of $\epsilon$ small
enough to work at every point.  Now, choose a unit vector $V\in T_XM$ at 
random, independent of $X$, and define $$X_\epsilon=exp_X(\epsilon V).$$
To see that this defines an exchangeable pair of random points on $M$, 
we show that 
$$\int_{M\times M}g(x,x_{\epsilon})d\mu=\int_{M\times M}
g(x_{\epsilon},x)d\mu$$
for all integrable $g:M\times M\to\R$, where $\mu$ is the measure 
defined by the construction of the pair $(X,X_\epsilon)$.

Let $L$ denote the normalized Liouville measure on $SM$, which is locally the 
product of the normalized volume measure on $M$ and normalized Lebesgue 
measure on the unit spheres of $T_xM$.  The measure $L$ has the 
property that it is invariant under the geodesic flow $\Phi^t(x,v)$.  See
\cite{chavel}, section 5.1 for a construction of $L$ and 
proofs of its key properties.
By construction of the measure $\mu$, 
\begin{equation*}\begin{split}
\int_{M\times M}g(x,x_{\epsilon})d\mu&=\int_{SM}g(x,\pi\circ\Phi^\epsilon(x,v))
dL(x,v)\\&=\int_{SM}g(\pi\circ\Phi^{-\epsilon}(y,\eta),y)dL(y,\eta)
\end{split}\end{equation*}
by the substitution $(y,\eta)=\Phi^\epsilon(x,v)$, since 
Liouville measure is invariant under the geodesic flow.  Now,
$$\Phi^{-\epsilon}(y,\eta)=-\Phi^\epsilon(y,-\eta),$$
as both correspond to going backwards on the same geodesic.
It follows that 
\begin{equation*}\begin{split}
\int_{SM}g(x,x_{\epsilon})dL(x,v)&=\int_{SM}g(\pi(-\Phi^{\epsilon}
(y,-\eta)),y)dL(y,\eta)\\&=\int_{SM}g(\pi(-\Phi^{\epsilon}(y,\eta)),y)
dL(y,\eta)\\&=\int_{SM}g(\pi(\Phi^{\epsilon}(y,\eta)),y)
dL(y,\eta)\\
&=\int_{M\times M}g(y_\epsilon,y)d\mu,
\end{split}\end{equation*}
where the second line follows from the substitution $-\eta\to\eta$, 
under which Lebesgue measure on the tangent space is invariant, and the 
third line follows because $\pi(-v)=\pi(v)$ (both have the same base point).
Thus the pair $(X,X_\epsilon)$
is exchangeable as required.

Now, let $f(x)=\sum a_if_i(x)$ as in the statement of the theorem.
For notational convenience, let $W=f(X)$ and 
$W_\epsilon=f(X_\epsilon)$.  Since $(X,X_\epsilon)$ 
is exchangeable, $(W,W_\epsilon)$ is an
exchangeable pair as well.  

In order to verify the conditions of Theorem \ref{abscont}, first 
let $\gamma:[0,
\epsilon] \to M$ be a constant-speed geodesic such that $\gamma(0)=X$, 
$\gamma(\epsilon)=
X_\epsilon$, and $\gamma'(0)=V$.  
Then applying Taylor's theorem on $\R$ to
the function $f\circ\gamma$ yields
\begin{equation}\begin{split}\label{taylor}
f(X_\epsilon)-f(X)&=\epsilon\cdot\left.\frac{d(f\circ\gamma)}{dt}\right|_{t=0}+
\frac{\epsilon^2}{2}\cdot\left.\frac{d^2(f\circ\gamma)}{dt^2}\right|_{t=0}
+O(\epsilon^3)\\
&=\epsilon\cdot d_Xf(V)+\frac{\epsilon^2}{2}
\cdot\left.\frac{d^2(f\circ\gamma)}{dt^2}
\right|_{t=0}+O(\epsilon^3),
\end{split}\end{equation}  
where the coefficient implicit in the $O(\epsilon^3)$ depends on $f$ and 
$\gamma$ and $d_xf$ denotes the differential of $f$ at $x$.  Recall that
$d_xf(v)=\inprod{\nabla f(x)}{v}$ for $v\in T_xM$ and the gradient $
\nabla f(x)$ defined as above.    
Now, for $X$ fixed, 
$V$ is distributed according to normalized Lebesgue measure on
$S_XM$ and $d_Xf$ is a linear functional on $T_XM$.  It follows that
$$\E\left[d_Xf(V)\big|X\right]=\E\left[d_Xf(-V)\big|X\right]=
-\E\left[d_Xf(V)\big|X\right],$$
so $$\E\left[d_Xf(V)\big|X\right]=0.$$
This implies that
$$\lim_{\epsilon\to0}\frac{1}{\epsilon^2}\E\big[f(X_\epsilon)-f(X)\big|
X\big]$$
exists and is finite.  
Indeed, it is well-known that 
 $$\lim_{\epsilon\to0}\frac{1}{\epsilon^2}\E\big[f(X_\epsilon)
-f(X)\big|X\big]=\frac{1}{2n}\Delta_gf(X)$$
for $n=dim(M),$ thus $$\lim_{\epsilon\to0}\frac{1}{\epsilon^2}
\E\left[W_\epsilon-W\big|W\right]=
-\frac{1}{2n}\E\left[\sum_{i=1}^ma_i\mu_if_i(X)\big|W\right]=
-\frac{\mu}{2n}W+\frac{1}{2n}\E\left[\left.\sum_{i=1}^ma_i(\mu-\mu_i)f_i(X)
\right|W\right].$$
It follows that
$\lambda = \frac{\mu}{2n}$ and
$E'=\frac{1}{2n}\E\left[\left.\sum_ia_i(\mu-\mu_i)f_i(X)\right|W\right].$
The higher order terms in $\epsilon$ of $\E\left[W_\epsilon-W\big|W\right]$
satisfy the integrability requirement of Theorem \ref{abscont} since
$f$ is smooth and $M$ is compact.
Finally,
\begin{equation*}\begin{split}
\frac{1}{\lambda}\E|E'|&\le\frac{1}{\mu}\E\left|\sum_ia_i(\mu_i-\mu)f_i(X)
\right|\\&\le\frac{1}{\mu}
\sqrt{\E\left[\sum_ia_i(\mu_i-\mu)f_i(X)\right]^2}\\&=\frac{1}{\mu}
\sqrt{\sum_ia_i^2(\mu_i-\mu)^2},
\end{split}\end{equation*}
since the $f_i$ are mutually orthogonal.

\medskip

\medskip

Consider next
$$\lim_{\epsilon\to0}\frac{1}{\epsilon^2}\E\big[(W_\epsilon-W)^2\big|X\big].$$
By the expansion (\ref{taylor}),
\begin{eqnarray*}
\E\left[(W_\epsilon-W)^2\big|X\right]&=&\E\left[(f(X_\epsilon)-f(X))^2\big|X
\right]\\&=&\epsilon^2\E\left[(d_Xf(V))^2\big|X\right]+O(\epsilon^3).
\end{eqnarray*}
Choose coordinates $\left\{\frac{\partial}{\partial x_i}\right\}_{i=1}^n$
in a neighborhood of $X$ which are orthonormal at $X$.  Then
$$\nabla f(X)=\sum_i\frac{\partial f}{\partial x_i}\frac{\partial}
{\partial x_i},$$
thus
\begin{equation*}\begin{split}
\left[d_xf(v)\right]^2&=\left[\inprod{\nabla f}{v}\right]^2\\
&=\sum_{i=1}^n\left(\frac{\partial f}{\partial x_i}\right)^2v_i^2+
\sum_{i\neq j}\frac{\partial f}{\partial x_i}\frac{\partial f}{\partial x_j}
v_iv_j.
\end{split}\end{equation*}
Since $V$ is uniformly distributed on a Euclidean sphere, 
$\E[V_iV_j]=\frac{1}{n}\delta_{ij}$.   Making use of this fact yields
$$\lim_{\epsilon\to0}\frac{1}{\epsilon^2}\E\left[(d_Xf(V))^2\big|X\right]=
\frac{1}{n}\|\nabla f\|^2=\frac{\mu}{n}+\frac{1}{n}\left[\|\nabla f\|^2
-\mu\right],$$
thus condition (2) is satisfied with 
$$E=\frac{1}{n}\E\left[\left.\|\nabla f(X)\|^2-\mu\right|W\right].$$

By Stokes' theorem,
\begin{equation*}\begin{split}
\E\|\nabla f(X)\|^2&=-\E\big[f(X)\Delta_g f(X)\big]\\
&=\E\left[f(X)\sum_ia_i\mu_if_i(X)\right]\\&=\mu\E f^2(X)+\E\left[f(X)\sum_i
a_i(\mu_i-\mu)f_i(X)\right]\\&=\mu+\E\left[f(X)\sum_i
a_i(\mu_i-\mu)f_i(X)\right]
\end{split}\end{equation*}
Thus 
\begin{equation*}\begin{split}
E&=\frac{1}{n}\E\left[\left.\|\nabla f\|^2-\mu-\E\left[f(X)\sum_i
a_i(\mu_i-\mu)f_i(X)\right]+\E\left[f(X)\sum_ia_i(\mu_i-\mu)f_i(X)\right]
\right|W\right]\\
&=\frac{1}{n}\left[\E\left[\left.\|\nabla f\|^2-\E\|\nabla f\|^2
\right|W\right]+\E\left[f(X)\sum_ia_i(\mu_i
-\mu)f_i(X)\right]\right],
\end{split}\end{equation*}
and so
\begin{equation*}\begin{split}
\frac{1}{\lambda}\E|E|&\le\frac{2}{\mu}\left[\E\left|\|\nabla f(X)\|^2-
\E\|\nabla f(X)\|^2\phantom{\sum_i}\hspace{-5.5mm}\right|
+\E\left|f(X)\sum_ia_i(\mu_i-\mu)f_i(X)\right|\right]\\
&\le\frac{2}{\mu}\left[\E\left|\|\nabla f(X)\|^2-
\E\|\nabla f(X)\|^2\phantom{\sum_i}\hspace{-5.5mm}\right|
+\sqrt{\E f^2(X)}
\sqrt{\E\left(\sum_ia_i(\mu_i-\mu)f_i(X)\right)^2}\right]\\&\le
\frac{2}{\mu}\left[\E\left|\|\nabla f(X)\|^2-
\E\|\nabla f(X)\|^2\phantom{\sum_i}\hspace{-5.5mm}\right|
+\sqrt{\sum_ia_i^2(\mu_i-\mu)^2}\right]
\end{split}\end{equation*}

Finally, (\ref{taylor}) 
gives immediately that 
$$\E\left[|W_\epsilon-W|^3\big|W\right]
=O(\epsilon^3).$$

\end{proof}

\section{The sphere}\label{sphere}
Eigenvalues and eigenfunctions of the Laplacian on the sphere $\s^{n-1}$ are 
well-studied objects.  In particular, it is discussed in section 9.5 of 
\cite{vil} 
that the eigenfunctions of 
the spherical Laplacian are exactly the restrictions of homogeneous harmonic
polynomials on $\R^n$ to the sphere, and that such polynomials of 
degree $l$ have eigenvalue $-l(l+n-2).$   It follows, then, that the simplest
class of functions on the sphere to which one could try to apply Theorem
\ref{absmfld} are the linear functions.  As was mentioned in the 
introduction, understanding the value distribution of linear functions
on the sphere is an old problem.  A careful history is given in 
the paper of Diaconis and Freedman \cite{diafree}, in which the following
theorem (and more general versions) are proved.
\begin{thm}[Diaconis-Freedman]
Let $X$ be a random point on the sphere $\sqrt{n}\s^{n-1}\subseteq\R^n$.
Then $$d_{TV}(X_1,Z)\le\frac{4}{n-1},$$
where $Z\sim\n(0,1).$
\end{thm}
It is straightforward to reproduce this result as a consequence of 
Theorem \ref{absmfld}.  In Section \ref{harmonic} below, 
Theorem \ref{absmfld} is applied to the
next obvious example, the second order spherical harmonics.  Section
\ref{degl} treats more specialized examples of spherical harmonics of 
arbitrary (odd) degree.

The following lemma, taken from \cite{folland},
for integrating polynomials over the sphere will be 
useful in applications.

\begin{lemma}\label{sph-ints}
Let $P(x)=|x_1|^{\alpha_1}|x_2|^{\alpha_2}\cdots|x_n|^{\alpha_n}$.  
Then if $X$ is uniformly distributed on $S^{n-1}$, 
$$\E\big[P(X)\big]=\frac{\Gamma(\beta_1)\cdots\Gamma(\beta_n)\Gamma(\frac{n}{2})}{\Gamma(\beta_1+\cdots+\beta_n)\pi^{n/2}},$$
where $\beta_i=\frac{1}{2}(\alpha_i+1)$ for $1\le i\le n$ and 
$$\Gamma(t)=\int_0^\infty s^{t-1}e^{-s}ds=2\int_0^\infty r^{2t-1}e^{-r^2}dr.$$
\end{lemma}

\subsection{Second order spherical harmonics}\label{harmonic}

\begin{thm}\label{sph}
Let $g(x)=\sum_{i,j}a_{ij}x_ix_j,$ where $A=(a_{ij})_{i,j=1}^n$ is a
symmetric matrix with $\tr(A)=0$.  Let $f=Cg$, where $C$ is chosen such that
$\|f\|_2=1$ when $f$ is considered as a function on $S^{n-1}$.  Let $W=f(X),$ 
where $X$ is a random point on $S^{n-1}$.  If $d$ is the 
vector in $\R^n$ whose entries are the eigenvalues of $A$, then 
\begin{equation*}
d_{TV}(W,Z)\le \sqrt{6}\left(\frac{\|d\|_4}{\|d\|_2}\right)^2,\end{equation*}
where $\|d\|_p=\left(\sum_i|d_i|^p\right)^{1/p}.$
 
\end{thm}

\medskip

\begin{proof}
To apply Theorem \ref{absmfld}, 
first note that $f$ is indeed 
an eigenfunction of the Laplacian on $S^{n-1}$;since
$A$ is traceless, $g$ is harmonic on $\R^n$ and thus is an eigenfunction
with eigenvalue $\lambda=-2n$.

Next, observe that $g(x)=\inprod{x}{Ax},$ and $A$ is a symmetric matrix, so
$A=U^*DU$ for a diagonal matrix $D$ and an orthogonal matrix $U$.  
Thus $$g(x)=\inprod{x}{U^*DUx}=\inprod{Ux}{DUx}.$$
Since the uniform distribution on the sphere is invariant under the 
action of the orthogonal group, this observation means that it suffices
to prove the theorem in the case that $A=diag(a_1,\ldots,a_n).$

Since $f$ is an eigenfunction, the second error term of Theorem \ref{absmfld}
vanishes.  By the theorem and an application of H\"older's inequality,
$$d_{TV}(W,Z)\le\frac{1}{|\lambda|}\sqrt{Var\big(\|\nabla f\|^2\big)},$$
 and the main task of the proof is to estimate the right-hand side.
A consequence of Stokes' theorem is that 
$$\int_{S^{n-1}}\|\nabla g\|^2d\sigma=|\lambda|\int_{S^{n-1}}g^2d\sigma,$$
so one can calculate $\E\|\nabla g\|^2$ in order to determine $C$.
The spherical gradient $\nabla g$ is just the projection onto the hyperplane
orthogonal to the radial direction of the usual gradient.  Letting
$\sideset{}{'}\sum$ stand for summing over distinct indexes, 
\begin{eqnarray*}
\E\|\nabla_{S^{n-1}}g(x)\|^2&=&\E\|\nabla_{\R^n}g(x)\|^2-\E\left(x\cdot\nabla_
{\R^n}g(x)\right)^2\\
&=&\E\left[\sum_{i=1}^n4a_i^2x_i^2\right]-4\E\left[\sum_{i=1}^na_i^2x_i^4+
\sideset{}{'}\sum_{i,j}a_ia_jx_i^2x_j^2\right]\\
&=&\frac{4}{n}\sum_{i=1}^na_i^2-\frac{4}{n(n+2)}\left[3\sum_ia_i^2+
\sideset{}{'}\sum_{i,j}a_ia_j\right]\\
&=&\frac{4}{n}\sum_{i=1}^na_i^2-\frac{4}{n(n+2)}\left[2\sum_ia_i^2+
\left(\sum_{i=1}^na_i\right)^2\right]\\
&=&\frac{4}{n+2}\sum_{i=1}^na_i^2,
\end{eqnarray*}
where Lemma \ref{sph-ints} 
has been used to get the third line, and the fact that $\tr(A)=0$ is
used to get the last line.  From this computation it follows that the 
constant $C$ in the statement of the theorem should be taken to be
$\sqrt{\frac{n(n+2)}{2\|a\|_2^2}},$ where $a=(a_i)_{i=1}^n.$
   
Now,
\begin{equation*}
\E\|\nabla_{S^{n-1}}g\|^4\le\E\|\nabla_{\R^n}g\|^4,
\end{equation*}
since $\nabla_{S^{n-1}}g$ is a projection of $\nabla_{\R^n}g$.
\begin{eqnarray*}
\E\|\nabla_{\R^n}g\|^4&=&16\,\E\left[\sum_{i=1}^na_i^4x_i^4+\sideset{}{'}
\sum_{i,j}a_i^2a_j^2x_i^2x_j^2\right]\\
&=&\frac{16}{n(n+2)}\left[\left(\sum_ia_i^2\right)^2+2\sum_ia_i^4\right],
\end{eqnarray*}
and so
\begin{eqnarray*}
\E\|\nabla_{S^{n-1}}f\|^4&\le&\frac{n^2(n+2)^2}{4\|a\|_2^4}\E\|\nabla_
{\R^n}g\|^4\\&=&\frac{4n(n+2)}{\|a\|_2^4}\left[\left(\sum_ia_i^2\right)^2+
2\sum_ia_i^4\right]\\
&=&4n(n+2)\left[1+\frac{2\|a\|_4^4}{\|a\|_2^4}\right].
\end{eqnarray*}
This gives that
$$\var\left(\|\nabla f\|^2\right)\le 8n^2\left(\frac{\|a\|_4^4}{\|a\|_2^4}
\right)+8n\left[1+\frac{2\|a\|_4^4}{\|a\|_2^4}\right],$$
thus
$$d_{TV}(W,Z)\le
\sqrt{\left(2+\frac{4}{n}\right)\left(\frac{\|a\|_4}{\|a\|_2}\right)^4
+\frac{2}{n}}\le\sqrt{4+\frac{4}{n}}\left(\frac{\|a\|_4}{\|a\|_2}\right)^2,$$
since $\|a\|_4\ge n^{-1/4}\|a\|_2$.
\end{proof}

\medskip

{\bf Remark:} 
For $n$ even, consider the quadratic function
$$f(x)=\sqrt{\frac{n+2}{2}}\left[\sum_{i=1}^{\frac{n}{2}}x_i^2-\sum_{j=
\frac{n}{2}+1}^nx_j^2\right].$$
This $f$ satisfies the conditions of the theorem and is normalized
such that $\E f^2(X)=1$ for $X$ a random point on $S^{n-1}$.  Theorem
\ref{sph} gives that
$$d_{TV}(f(X),Z)\le\sqrt{\frac{6}{n}},$$
so in this case, Theorem \ref{sph} gives a rate of convergence to normal.

The cases in which
the bound does not tend to zero as $n\to\infty$  are those in
which a small number of coordinates of $X$ control the value of $f(X)$; in
those situation one would not expect $f$ to be normally
distributed.  For example, if $f(x)=cx_1^2-cx_2^2,$ where $c$
is the proper normalization constant, then because $x_1$ and $x_2$
are asymptotically independent Gaussian random variables (see
\cite{diafree}), $f$ is asymptotically 
distributed as $c(Z_1^2-Z_2^2)$ for independent Gaussians $Z_1$ 
and $Z_2$.

\subsection{Higher degree spherical harmonics}\label{degl}
Fix an odd integer $\ell$.  A canonical example of an eigenfunction
of the spherical Laplacian of degree $\ell$ is the polynomial
$p_\ell(x)$ given by
$$p_\ell(x)=AC^{\frac{n-2}{2}}_\ell(x_n),$$
where $A^2=\frac{(n-3)!\ell!(n+2\ell-2)}{(n+\ell-3)!(n-2)}$ is a normalization
constant chosen such that $\E p_\ell^2(X)=1$ for $X$ a random point of 
$S^{n-1}$, and 
\begin{equation}\label{gegen}
C^k_\ell(t)=\frac{2^\ell}{\Gamma(k)}\sum_{j=0}^{\lfloor \ell/2\rfloor}\frac{
(-1)^j\Gamma(k+\ell-j)}{2^{2j}j!(\ell-2j)!}t^{\ell-2j}\end{equation}
is a Gegenbauer polynomial.  The eigenvalue of this eigenfunction of 
the Laplacian on $S^{n-1}$ is $\ell(n+\ell-2).$

The reason that this particular eigenfunction is natural is that
is arises in the
representation theory of the orthogonal group.  Let $T^{n\ell}$ be the 
irreducible representation of $SO(n)$ in the space
$\r^{n,\ell}/r^2\r^{n,\ell-2}$, where $\r^{n,\ell}$ is the space of 
homogeneous polynomials of degree $\ell$ in $n$ variables.  
The polynomial $p_\ell(x)$ above is the zonal polynomial of degree $\ell$ of 
the 
representation $T^{n\ell}$ relative to the subgroup $SO(n-1)$, where as usual,
$SO(n-1)$ is identified with the subgroup of $SO(n)$ fixing the north pole of
$S^{n-1}$. 
  See \cite{vil} for 
background on spherical harmonics and Gegenbauer polynomials, including all
facts stated above.

Observe that since 
$p_\ell^{(n)}(x)=AC^{\frac{n-2}{2}}_\ell(x_n)$ is an eigenfunction, so
are $p_\ell^{(k)}(x)=AC^{\frac{n-2}{2}}_\ell(x_k)$ for all $1\le k\le n$.  
Let $a_1,\ldots,a_n\subseteq\R$ such that $\sum a_i^2=1$, and consider
the degree $\ell$ eigenfunction
$$p(x)=\sum_{k=1}^na_kp_\ell^{(k)}(x).$$
Note that $\E p^2(X)=1$ for $X$ a random point of $S^{n-1}$, because if
$\ell$ is odd, then the $C^{\frac{n-2}{2}}_\ell(x_k)$ are orthogonal for
different $k$.  In this example, a bound on the total variation distance 
between
the distribution of the random variable $p(X)$ for $X$ random on the sphere
and a standard Gaussian random variable $Z$ is obtained in terms of the $a_i$.
Specifically, 
\begin{thm}
There is a constant $c$, depending on $\ell$, such that 
for $p:\s^{n-1}\to\R$ defined by 
$$\sum_{i=1}^na_ip_\ell^{(i)}(x)$$ 
with $\sum a_i^2=1$, 
$$d_{TV}(p(X),Z)\le c\|{\bf a}\|_4^2.$$
If the vector of coefficients $\{a_i\}$ is chosen to be random and 
uniformly distribtuted on the sphere, then there is another constant $c'$
such that 
$$\E d_{TV}(p(X),Z)\le \frac{c'}{\sqrt{n}}.$$
\end{thm}

\medskip

{\bf Remark: }The second statement of the theorem says that, typically, 
spherical harmonics obtained as linear combinations of the $p_\ell^{(i)}(x)$
are approximately Gaussian in distribution.  As in Theorem \ref{sph}, 
it may also be helpful to consider a specific
example.  If $|a_i|=\frac{1}{\sqrt{n}}$ for each $i$, then 
$$d_{TV}(p(X),Z)\le\frac{c}{\sqrt{n}}$$
for a constant $c$ depending only on $\ell$.  As discussed 
after the statment of Theorem \ref{sph}, the bound $\|{\bf a}\|_4^2$ is
not small in the case that only a fixed small number $k$ of $a_i$ are non-zero.
In such a case, $p(X)$ is not distributed as a Gaussian random variable
but as a degree $\ell$ polynomial in $k$ independent Gaussian random 
variables.

\begin{proof}
Recall that Theorem \ref{absmfld} gives 
\begin{equation}\begin{split}\label{error}
d_{TV}(p(X),Z)&\le\frac{1}{\mu}\sqrt{
\var(\|\nabla_{S^{n-1}}p(X)\|^2)}\\&=
\frac{1}{\ell(n+\ell-2)}\sqrt{\E\|\nabla_{S^{n-1}}p(X)\|^4-\ell^2(n+\ell-2)^2},
\end{split}\end{equation}
since $$\E\|\nabla_{S^{n-1}}p(X)\|^2=\ell(n+\ell-2)$$
by Stokes' theorem.
Thus to apply the theorem, we must compute (or bound) 
$$\E\|\nabla_{S^{n-1}}p(X)\|^4.$$
Consider the function $p(x)=\sum_{k=1}^na_kp_
\ell^{(k)}(x)$ on a neighborhood of $S^{n-1}$.  Then for $x\in S^{n-1}$,
$$\|\nabla_{S^{n-1}}p(x)\|^2=\|\nabla p(x)\|^2-\inprod{x}{\nabla p(x)}^2,$$
where $\nabla p(x)$ denotes the usual $\R^n$-gradient of $p$.  Thus
$$\E\|\nabla_{S^{n-1}}p(X)\|^4\le\E\|\nabla p(X)\|^4;$$
bounding the right hand side turns out to be enough.

It is proved in \cite{vil} that the Gegenbauer polynomials 
have the property that
\begin{equation}\label{diff}
\frac{d}{dt}C^p_m(t)=2pC^{p+1}_{m-1}(t),
\end{equation} 
thus
$$\nabla p(x)=\left(Aa_k(n-2)C^{\frac{n}{2}}_{\ell-1}(x_k)\right)_{k=1}^n.$$
It is therefore necessary to estimate
\begin{equation}\label{norm4}\begin{split}
\E&\left[A^2(n-2)^2\sum_{k=1}^na_k^2(C^{\frac{n}{2}}_{\ell-1}(x_k))^2\right]^2
\\&\qquad=A^4(n-2)^4\E\left[\sum_{k=1}^na_k^4
(C^{\frac{n}{2}}_{\ell-1}(x_k))^4+\sideset{}{'}\sum_{k,j=1}^na_k^2a_j^2
(C^{\frac{n}{2}}_{\ell-1}(x_k))^2(C^{\frac{n}{2}}_{\ell-1}(x_j))^2\right],
\end{split}\end{equation}
where $\sideset{}{'}\sum$ stands for summing over distinct indices.

The proof hinges on the following key facts.
\begin{enumerate}
\item For $k\neq j$, 
$$A^4(n-2)^4\E\left[(C^{\frac{n}{2}}_{\ell-1}(x_k))^2(C^{\frac{n}{2}}_{
\ell-1}(x_j))^2\right]=\ell^2\big(n+O(1)\big)^2.$$
\item There is a constant $c$ such that for all $k$, 
$$A^4(n-2)^4\E\left[(C^{\frac{n}{2}}_{\ell-1}(x_k))^4\right]\le cn^2.$$
\end{enumerate}
Using the normalization $\sum a_k^2=1,$
these facts imply that 
$$\E\|\nabla_{S^{n-1}}p(X)\|^4\le\ell^2n^2+c'n^2\sum_{k=1}^n
a_k^4+c''n$$
for constants $c', c''$.  Note that since $\|{\bf a}\|_2=1,$ $\sum a_k^4\ge
\frac{1}{n}$ and so the last term can be absorbed into the middle term.
The result is now immediate from \eqref{error}.

It remains to verify the two key facts above.  By \eqref{gegen},
\begin{equation*}\begin{split}
\E&\left[C^{\frac{n}{2}}_{\ell-1}(X_1)C^{\frac{n}{2}}_{\ell-1}(X_2)\right]^2\\
&=\left(\frac{2^{\ell-1}}{\Gamma\left(\frac{n}{2}\right)}\right)^4\cdot\\&\quad
\sum_{k,m,p,q=0}^{\frac{\ell-1}{2}}\left(\frac{(-1)^{k+m+p+q}\Gamma\left(\ell-1+\frac{n}{2}-k\right)
\Gamma\left(\ell-1+\frac{n}{2}-m\right)\Gamma\left(\ell-1+\frac{n}{2}-p\right)
\Gamma\left(\ell-1+\frac{n}{2}-q\right)}{2^{2k+2m+2p+2q}k!m!p!q!(\ell-1-2k)!
(\ell-1-2m)!(\ell-1-2p)!(\ell-1-2q)!}\cdot\right.\\
&\qquad\qquad\qquad\left.\phantom{\frac{(-1)^k\left(\frac{n}{2}\right)}
{\frac{n}{2}}}\E\left[X_1^{2\ell-2k-2m-2}X_2^{2\ell-2p-2q-2}\right]\right).
\end{split}\end{equation*}
Lemma \ref{sph-ints} gives
\begin{equation*}\begin{split}
\E\Big[X_1^{2\ell-2k-2m-2}&X_2^{2\ell-2p-2q-2}\Big]\\&=
\frac{\Gamma\left(\ell-k-m-\frac{1}{2}\right)\Gamma\left(\ell-p-q-
\frac{1}{2}\right)\Gamma\left(\frac{n}{2}\right)}{\pi\Gamma\left(\frac{n}{2}
+2\ell-k-m-p-q-2\right)}\\
&=\frac{1}{\pi}\Gamma\left(\ell-k-m-\frac{1}{2}\right)\Gamma\left(\ell-p-q-
\frac{1}{2}\right)\left(\frac{2}{n}+O\left(\frac{1}{n^2}\right)\right)^{2\ell-k
-m-p-q-2},
\end{split}\end{equation*}
where here and throughout, $O\left(\frac{1}{n^2}\right),$ $O(1)$, etc. refer
to the limit as $n\to\infty$, with $\ell$ still fixed and the 
constants implicit in the $O$ depending on $\ell$.  It follows that
\begin{equation}\begin{split}\label{gamma-1}
\E&\left[C^{\frac{n}{2}}_{\ell-1}(X_1)C^{\frac{n}{2}}_{\ell-1}(X_2)\right]^2\\
&=\left(\frac{2^{2\ell-2}}{\Gamma^2\left(\frac{n}{2}\right)}\sum_{k,m=0}^{
\frac{\ell-1}{2}}\frac{(-1)^{k+m}\Gamma\left(\ell-1+\frac{n}{2}-k\right)
\Gamma\left(\ell-1+\frac{n}{2}-m\right)\Gamma\left(\ell-k-m-\frac{1}{2}
\right)}{2^{2k+2m}k!m!(\ell-1-2k)!(\ell-1-2m)!\sqrt{\pi}}\cdot\right.\\
&\qquad\qquad\qquad\qquad\qquad\qquad\left.\left(\frac{2}{n}+
O\left(\frac{1}{n^2}\right)\right)^{\ell-k-m-1}\right)^2.
\end{split}\end{equation}
Now,
\begin{equation}\label{gammas}
\Gamma\left(\ell-1+\frac{n}{2}-k\right)=\Gamma\left(\frac{n}{2}\right)
\left(\frac{n}{2}+O(1)\right)^{\ell-k-1},\end{equation}
so the right-hand side of \eqref{gamma-1} reduces to
\begin{equation*}
\left(\big(2n+O(1)\big)^{\ell-1}\sum_{k,m=0}^{\frac{\ell-1}{2}}\frac{(-1)^{k+m}
\Gamma\left(\ell-k-m-\frac{1}{2}\right)\pi^{-1/2}}
{2^{2k+2m}k!m!(\ell-1-2k)!(\ell-1-2m)!}\right)^2
\end{equation*}

Recall that 
$$A^2=\frac{(n-3)!\ell!(n+2\ell-2)}{(n+\ell-3)!(n-2)}=\ell!\big(n+O(1)\big)^{
-\ell},$$
thus
\begin{equation*}\begin{split}
A^4(n-2)^4\E&\left[C^{\frac{n}{2}}_{\ell-1}(X_1)C^{\frac{n}{2}}_{
\ell-1}(X_2)\right]^2\\&=(\ell!)^2\big(n+O(1)\big)^2\left(2^{\ell-1}
\sum_{k,m=0}^{\frac{\ell-1}{2}}\frac{(-1)^{k+m}
\Gamma\left(\ell-k-m-\frac{1}{2}\right)\pi^{-1/2}}
{2^{2k+2m}k!m!(\ell-1-2k)!(\ell-1-2m)!}\right)^2.
\end{split}\end{equation*}

{\bf Claim:} For $\ell\in\N$ odd, $$2^{\ell-1}
\sum_{k,m=0}^{\frac{\ell-1}{2}}\frac{(-1)^{k+m}
\Gamma\left(\ell-k-m-\frac{1}{2}\right)\pi^{-1/2}}
{2^{2k+2m}k!m!(\ell-1-2k)!(\ell-1-2m)!}=\frac{1}{(\ell-1)!}.$$

A proof of this claim is given in the
appendix.  Making use of the claim, it is 
immediate that
$$A^4(n-2)^4\E\left[C^{\frac{n}{2}}_{\ell-1}(X_1)C^{\frac{n}{2}}_{
\ell-1}(X_2)\right]^2=\ell^2\big(n+O(1)\big)^2.$$

\medskip

It remains to verify fact 2: there is a constant $c$ such that for all
$k$,
$$A^4(n-2)^4\E\left[(C^{\frac{n}{2}}_{\ell-1}(x_k))^4\right]\le cn^2.$$
In fact, this is much easier than the previous part since the exact asymptotic
constant is not needed.  Proceeding as before,
\begin{equation*}\begin{split}
\E&\left[C^{\frac{n}{2}}_{\ell-1}(X_1)\right]^4\\
&=\left(\frac{2^{\ell-1}}{\Gamma\left(\frac{n}{2}\right)}\right)^4\cdot\\&\quad
\sum_{k,m,p,q=0}^{\frac{\ell-1}{2}}\left(\frac{(-1)^{k+m+p+q}\Gamma\left(\ell-1+\frac{n}{2}-k\right)
\Gamma\left(\ell-1+\frac{n}{2}-m\right)\Gamma\left(\ell-1+\frac{n}{2}-p\right)
\Gamma\left(\ell-1+\frac{n}{2}-q\right)}{2^{2k+2m+2p+2q}k!m!p!q!(\ell-1-2k)!
(\ell-1-2m)!(\ell-1-2p)!(\ell-1-2q)!}\cdot\right.\\
&\qquad\qquad\qquad\left.\phantom{\frac{(-1)^k\left(\frac{n}{2}\right)}
{\frac{n}{2}}}\E\left[X_1^{4\ell-2k-2m-2p-2q-4}\right]\right).
\end{split}\end{equation*}
From Lemma \ref{sph-ints},
\begin{equation*}\begin{split}
\E\Big[X_1^{4\ell-2k-2m-2p-2q-4}\Big]&=
\frac{\Gamma\left(2\ell-k-m-p-q-\frac{3}{2}\right)
\Gamma\left(\frac{n}{2}\right)}{\pi\Gamma\left(\frac{n}{2}
+2\ell-k-m-p-q-2\right)}\\
&=\frac{1}{\pi}\Gamma\left(2\ell-k-m-p-q-\frac{3}{2}\right)
\left(\frac{2}{n}+O\left(\frac{1}{n^2}\right)\right)^{2\ell-k
-m-p-q-2},
\end{split}\end{equation*}
and so
\begin{equation*}\begin{split}
\E&\left[C^{\frac{n}{2}}_{\ell-1}(X_1)\right]^4\\
&=\left(2^{4\ell-4}\sum_{k,m,p,q=0}^{
\frac{\ell-1}{2}}\frac{(-1)^{k+m+p+q}\Gamma\left(2\ell-k-m-p-q-\frac{3}{2}
\right)\left(\frac{n}{2}+
O\left(1\right)\right)^{2\ell-2}}{2^{2k+2m+2p+2q}k!m!p!q!(\ell-1-2k)!(\ell-1-2m)!(\ell-1-2p)!
(\ell-1-2q)!\pi}\right),
\end{split}\end{equation*}
making use again of \eqref{gammas}.
Since 
$$A^4=(\ell!)^2\left(n+O(1)\right)^{-2\ell},$$
it follows that 
$$A^4(n-2)^4\E\left[C^{\frac{n}{2}}_{\ell-1}(X_1)\right]^4=\big(n+O(1)\big)^2
\cdot C,$$
where $C$ (which depends on $\ell$) is independent of $n$. 
\end{proof}

\section{The torus}\label{torus}
This section deals with the value distributions of eigenfunctions
on flat tori.  The class of functions considered here are random functions;
that is, they are linear combinations of eigenfunctions with random
coefficients.

Let $(M,g)$ be the torus $\T^n=\R^n/\Z^n$, with the metric
given by the symmetric positive-definite bilinear form $B$: 
$$(x,y)_B=\inprod{Bx}{y}.$$  With this metric, the Laplacian $\Delta_B$
on $\T^n$ is given by
$$\Delta_Bf(x)=\sum_{j,k}(B^{-1})_{jk}\frac{\partial^2f}{\partial x_j\partial
x_k}(x).$$
Eigenfunctions of $\Delta_B$ are given by the real and imaginary parts
of functions of the form 
$$f_v(x)=e^{2\pi i(v,x)_B}=e^{2\pi i\inprod{Bv}{x}},$$
for vectors $v\in\R^n$ such that $Bv$ has integer components, with
corresponding eigenvalue $-\mu_v=-(2\pi \|v\|_B)^2.$

Consider a linear combination of eigenfunctions with random
coefficients:
$$f(x):=\Re\left(\sum_{v\in\V}a_ve^{2\pi i\inprod{Bv}{x}}\right),$$
where $\V$ is a finite collection of vectors $v$ such that 
$Bv$ has integer components for each $v\in\V$ and
$\{a_v\}_{v\in\V}$ is a random vector on the sphere of radius
$\sqrt{2}$ in $\R^\V$.  Assume further that $v+w\neq 0$ for $v,w\in\V$.
Then the following theorem holds.
\begin{thm}\label{torusthm}
For a random function $f$ on $(\T^n,B)$ defined as above and for $\mu>0$,
$$\E d_{TV}(f(X),Z)\le\frac{1}{|\mu|}\left[\sqrt{\frac{8(2\pi)^4}
{|\V|(|\V|+2)}\sum_{v,w\in\V}(v,w)_B^2}+
\left(2\sqrt{2}+\sqrt{\pi}\right)\sqrt{\frac{1}{|\V|}
\sum_{v\in\V}(\mu_v-\mu)^2}\right].$$
\end{thm}

The idea is that the eigenvalues $-\mu_v=-(2\pi \|v\|_B)^2$ 
corresponding to the
summands be clustered about the value $\mu$ and that the size of 
$\V$ be large.  In
some cases, e.g. for $B=I$, it may be possible to eliminate the 
second error term by choosing a large set $\V$
such that all elements have the same $B$-norm, that is, it may be that 
the eigenspaces of the Laplacian $\Delta_B$ have large dimension.  However,
this need not be true in general.  The presence of the second error 
term is to take into account that while it may not be the case that eigenspaces
are large, it may be that there are many eigenspaces of low dimension whose 
eigenvalues are all close together.

Note that requiring the first term of the error to be small is 
a weakening of an orthogonality requirement on the elements of $\V$.
By the Cauchy-Schwarz inequality, 
$$\sum_{v,w\in\V}(v,w)_B^2\le\sum_{v,w\in\V}\|v\|_B^2\|w\|_B^2\le
\left(\max_{v\in\V}\|v\|_B^4\right)|\V|^2.$$
If all of the vectors $v\in\V$ have $\mu_v\approx\mu,$ then this 
last expression is approximately $\frac{\mu^2}{(2\pi)^4}|\V|^2,$
and the first error term is bounded by $2\sqrt{2}$.  This is of course
worthless, since total variation distance is always bounded by one.  On the
other hand, if the elements of $\V$ are mutually orthogonal, then 
$$\sum_{v,w\in\V}(v,w)_B^2=\sum_{v\in\V}\|v\|_B^4\approx\frac{\mu^2}{(2\pi)^4}
|\V|,$$
and then the first error term is bounded by $\frac{2\sqrt{2}}{\sqrt{|\V|}},$
which is small as long as $\V$ is large.  The idea is to aim for a 
set $\V$ somewhere in between these two extremes: not to require the 
elements of $\V$ to be orthogonal, but to choose them to be fairly
evenly distributed on the $B$-norm sphere on which they lie.

\noindent{\bf Examples.}
\begin{enumerate}
\item Let $B=I$.  Choose $\V$ to be the set of vectors in $\R^n$ with 
two entries equal to one and all others zero.  Then $|\V|=\binom{n}{2}$  
and $\|v\|^2=2$ for each
$v\in\V$.  Given a fixed $v\in\V$, there are exactly $2n-3$ elements
$w\in\V$ (including $v$ itself) such that $\inprod{v}{w}\neq0$.  Thus
$$\frac{1}{|\V|(|\V|+2)}\sum_{v,w\in\V}\inprod{v}{w}^2\le\frac{1}{\binom{n}{2}
\left(\binom{n}{2}+2\right)}\cdot\binom{n}{2}(4)(2n-3)=\frac{16n-12}{n^2-n+4},$$
thus in this case there is a constant $c$ such that
$$\E d_{TV}(f(X),Z)\le\frac{c}{\sqrt{n}}.$$

\item Let $B=diag(1+\delta_1,\ldots,1+\delta_n)$, with $\delta_i$ small
real numbers.  Let $\V$ be the set of vectors of the 
previous example.  For $v\in\V$, let $v'=\left(\frac{1}{1+\delta_1}v_1,
\ldots,\frac{1}{1+\delta_n}v_n\right)$; $v'$ is constructed such that
$Bv'$ has integer components.  Let $\V'=\{v':v\in\V\}.$

Now, let $\epsilon=\max_i\left|1-\frac{1}{1+\delta_i}\right|$.  Then
\begin{equation*}
\left|\|v'\|_B^2-2\right|=\left|\sum_{i=1}^n\frac{1}{1+\delta_i}v_i^2-
\sum_{i=1}^nv_i^2\right|\le\epsilon\|v\|^2=2\epsilon.
\end{equation*}
By construction, $(v,w)_B\neq 0$ if and only if $\inprod{v}{w}\neq 0$, and
in those cases, $$(v,w)_B^2\le\|v\|_B^2\|w\|_B^2\le (2+2\epsilon)^4.$$
Taking $\mu=2(2\pi)^2$ in the theorem, it follows that for $f$
defined as above using this matrix $B$ and index set $\V'$, there
is a constant $C$ so that
$$\E d_{TV}(f(X),Z)\le C\max\left(\frac{1}{\sqrt{n}},\epsilon\right).$$

\end{enumerate}

\medskip

\begin{proof}[Proof of Theorem \ref{torusthm}]
To apply Theorem \ref{absmfld}, an expression for $\|\nabla_B f\|_B^2$ 
is needed.  
\begin{equation*}\begin{split}
\nabla_Bf(x)&=\left\{\Re\left(\sum_{j=1}^n\sum_{v\in\V}(2\pi i)a_v(Bv)_j
(B^{-1})_{jk}e^{2\pi i\inprod{Bv}{x}}\right)\right\}_{k=1}^n\\
&=-\Im\left(\sum_{v\in\V}(2\pi)a_ve^{2\pi i\inprod{Bv}{x}}v\right),
\end{split}\end{equation*}
using the fact that $B$ is symmetric.

It follows that
\begin{equation}\begin{split}
\|\nabla_Bf(x)\|_B^2&=\sum_{j,k}B_{jk}\left(\nabla_Bf(x)\right)_j\left(
\nabla_Bf(x)\right)_k\\&=\sum_{j,k}B_{jk}\Im\left(\sum_{v\in\V}(2\pi)a_v
e^{2\pi i\inprod{Bv}{x}}v_j\right)\Im\left(\sum_{w\in\V}(2\pi)a_we^{2\pi i
\inprod{Bw}{x}}w_k\right)\\&=\frac{1}{2}\Re\left[\sum_{v,w\in\V}4\pi^2a_va_w
\left(\sum_{j,k}B_{jk}v_jw_k\right)\left(e^{2\pi i\inprod{Bv-Bw}{x}}-
e^{2\pi i\inprod{Bv+Bw}{x}}\right)\right]\\&=\frac{1}{2}\Re\left[\sum_{v,w\in
\V}4\pi^2a_va_w(v,w)_B\left(e^{2\pi i\inprod{Bv-Bw}{x}}-
e^{2\pi i\inprod{Bv+Bw}{x}}\right)\right].\label{grad2}
\end{split}\end{equation}
Let $X$ be a randomly distributed point on the torus.  
The first half of the error term in Theorem \ref{absmfld} is
\begin{equation*}\begin{split}
\E\left|\|\nabla_B f(X)\|_B^2-\E_X\|\nabla_B f(X)\|_B^2\phantom{\sum_i}
\hspace{-5.5mm}\right|&\le\sqrt{\E\|\nabla_Bf(X)\|_B^4-\E\left(\E_X
\|\nabla_Bf(X)\|_B^2\right)^2}.
\end{split}\end{equation*}
Start by computing $\E\|\nabla_Bf(X)\|_B^4.$  From above,
\begin{equation*}\begin{split}
\|\nabla_Bf(x)\|_B^4&=2\pi^4\Re\left[\sum_{v,w,v',w'}a_va_wa_{v'}a_{w'}
(v,w)_B(v',w')_B\right.\\&
\qquad\qquad\qquad\Big[e^{2\pi i\inprod{Bv-Bw-Bv'+Bw'}{x}}-e^{2\pi i\inprod{
Bv-Bw-Bv'-Bw'}{x}}+e^{2\pi i\inprod{Bv-Bw+Bv'-Bw'}{x}}\\&
\qquad\qquad\qquad-e^{2\pi i\inprod{
Bv-Bw+Bv'+Bw'}{x}}-e^{2\pi i\inprod{Bv+Bw-Bv'+Bw'}{x}}+e^{2\pi i\inprod{
Bv+Bw-Bv'-Bw'}{x}}\\&\left.\qquad\qquad\qquad\phantom{\sum_v}
-e^{2\pi i\inprod{Bv+Bw+Bv'-Bw'}{x}}+e^{2\pi i\inprod{
Bv+Bw+Bv'+Bw'}{x}}\Big]\right].
\end{split}\end{equation*}
Averaging over the coefficients $\{a_v\}$ gives
\begin{equation*}\begin{split}
\E_a\|\nabla_Bf(x)\|_B^4=\frac{8\pi^4}{|\V|(|\V|+2)}\Re&\left[3
\sum_{v\in\V}\|v\|_B^4\left(3-e^{-4\pi i\inprod{Bv}{x}}-3e^{4\pi i\inprod{Bv}
{x}}+e^{8\pi i\inprod{Bv}{x}}\right)\right.\\&\quad+\sideset{}{'}
\sum_{v,v'\in\V}\|v\|_B^2\|v'\|_B^2\left(2-2e^{4\pi i\inprod{Bv}{x}}
-e^{-4\pi i\inprod{Bv'}{x}}-e^{4\pi i\inprod{Bv'}{x}}\right.\\
&\qquad\qquad\qquad\qquad\qquad\left.+e^{4\pi i\inprod{Bv-
Bv'}{x}}+e^{4\pi i\inprod{Bv+Bv'}{x}}\right)\\&\quad+2\sideset{}{'}\sum_{v,w
\in\V}(v,w)_B^2\left(2-2e^{4\pi i\inprod{Bv}{x}}-e^{-4\pi i\inprod{Bw}{x}}
-e^{4\pi i\inprod{Bw}{x}}\right.\\&\qquad\qquad\qquad\left.\phantom{
\sum_{v\in\V}}\left.
+e^{4\pi i\inprod{Bv-Bw}{x}}+e^{4\pi i\inprod{Bv+Bw}{x}}\right)\right].
\end{split}\end{equation*}
Now averaging over a random point $X$ of the torus gives
\begin{equation*}\begin{split}
\E\|\nabla_Bf(X)\|_B^4&=\frac{8\pi^4}{|\V|(|\V|+2)}\left[9\sum_{v\in\V}
\|v\|_B^4+2\sideset{}{'}\sum_{v,w\in\V}\|v\|_B^2\|v'\|_B^2+4\sideset{}{'}
\sum_{v,w\in\V}(v,w)_B^2\right]\\&=\frac{8\pi^4}{|\V|(|\V|+2)}\left[3
\sum_{v\in\V}\|v\|_B^4+2\left(\sum_{v\in\V}\|v\|_B^2\right)^2+4
\sum_{v,w\in\V}(v,w)_B^2\right].
\end{split}\end{equation*}

\medskip

Next,
\begin{equation*}\begin{split}
\E\Big[\E_X\|\nabla_Bf(X)\|_B^2\Big]^2&=
\E\left[2\pi^2\sum_{v\in\V}a_v^2\|v\|_B^2\right]^2\\&=\frac{(2\pi)^4}{|\V|
(|\V|+2)}\left[3\sum_{v\in\V}\|v\|_B^4+\sideset{}{'}\sum_{v,w\in\V}\|v\|_B^2
\|w\|_B^2\right]\\&=\frac{(2\pi)^4}{|\V|(|\V|+2)}\left[\left(\sum_{v\in\V}
\|v\|_B^2\right)^2+2\sum_{v\in\V}\|v\|_B^4\right],
\end{split}\end{equation*}
thus
\begin{equation*}\begin{split}
&\sqrt{\E\|\nabla_Bf(X)\|_B^4-\E\left(\E_X\|\nabla_Bf(X)\|_B^2\right)^2}
\\&
\qquad\qquad\qquad\le
\sqrt{\frac{(2\pi)^4}{|\V|(|\V|+2)}\left(2\sum_{v,w\in\V}(v,w)_B^2-\frac{1}{2}
\sum_{v\in\V}\|v\|_B^4\right)}\\&\qquad\qquad\qquad\le\sqrt{\frac{2(2\pi)^4}
{|\V|(|\V|+2)}\sum_{v,w\in\V}(v,w)_B^2}.
\end{split}\end{equation*}
\medskip

Finally, the second error term of Theorem \ref{absmfld} is
a multiple of $$\sqrt{\sum_{v\in\V}a_v^2(\mu_v-\mu)^2},$$
where $\mu_v=(2\pi\|v\|_B)^2$ and $\mu=\frac{1}{|\V|}\sum_{v\in\V}\mu_v.$
By H\"older's inequality,
$$\E\sqrt{\sum_{v\in\V}a_v^2(\mu_v-\mu)^2}\le\sqrt{\frac{2}{|\V|}
\sum_{v\in\V}(\mu_v-\mu)^2}.$$

This completes the proof.
\end{proof}

\section{Appendix}
This section contains the proof that for $\ell$ odd,
\begin{equation}\label{conjec}2^{\ell-1}
\sum_{k,m=0}^{\frac{\ell-1}{2}}\frac{(-1)^{k+m}
\Gamma\left(\ell-k-m-\frac{1}{2}\right)\pi^{-1/2}(\ell-1)!}
{2^{2k+2m}k!m!(\ell-1-2k)!(\ell-1-2m)!}=1.\end{equation}
For $n\in\N,$ let $n!!=n(n-2)(n-4)\cdots(1,2)$, where the last factor
is 1 or 2 depending on the parity of $n$.
Using the fact that $\Gamma(x)=(x-1)\Gamma(x-1)$ together with 
$\Gamma\left(\frac{1}{2}\right)=\sqrt{\pi},$ the left-hand side of 
\eqref{conjec} is 

\begin{equation}\begin{split}\label{nogamma}
\sum_{k,m=0}^{\frac{\ell-1}{2}}&\frac{(-1)^{k+m}(2\ell-2k-2m-3)!!(\ell-1)!}
{2^{k+m}k!m!(\ell-2k-1)!(\ell-2m-1)!}\\&=\sum_{k,m=0}^{\frac{\ell-1}{2}}
\frac{(-1)^{k+m}(2\ell-2k-2m-2)!(\ell-1)!}
{2^{k+m}k!m!(\ell-2k-1)!(\ell-2m-1)!(2\ell-2k-2m-2)!!}.
\end{split}\end{equation}
Next, let $p=\frac{\ell-1}{2}.$  Then \eqref{nogamma} becomes
\begin{equation*}\begin{split}
\sum_{k,m=0}^p&\frac{(-1)^{k+m}(4p-2k-2m)!(2p)!}{2^{k+m}k!m!(2p-2k)!(2p-2m)!
(4p-2k-2m)!!}\\&=2^{-2p}\sum_{k,m=0}^p\frac{(-1)^{k+m}(4p-2k-2m)!(2p)!}{
k!m!(2p-2k)!(2p-2m)!(2p-k-m)!}\\
&=2^{-2p}\sum_{k=0}^p\frac{(-1)^k(2p)!}{k!(2p-2k)!}\left(\sum_{m=0}^p
\frac{(-1)^m(4p-2k-2m)!}{m!(2p-2m)!(2p-k-m)!}\right),
\end{split}\end{equation*}
and the goal is to prove that this last expression is equal to 1.
We do this by showing 
\begin{equation}\label{hypergeom}
\sum_{m=0}^p\frac{(-1)^m(4p-2k-2m)!}{m!(2p-2m)!(2p-k-m)!}=\begin{cases}
2^{2p}&k=0\\0&k=1,\ldots,p.\end{cases}
\end{equation}
This can be done by a computer algebra system such as Maple; 
the following direct proof is due to Geir Helleloid \cite{geir}.

Let $(\alpha)_n$ be the shifted 
factorial
$$(\alpha)_n=\alpha(\alpha+1)\cdots(\alpha+n-1).$$
The following two straightforward identities involving shifted factorials 
will be useful:
$$(n-k)!=\frac{(-1)^kn!}{(-n)_k}$$
and
$$(a)_{2k}=2^{2k}\left(\frac{a}{2}\right)_k\left(\frac{a+1}{2}\right)_k.$$
The hypergeometric series
$$\sum_{n=0}^\infty\frac{(a_1)_n\cdots(a_p)_n}{(b_1)_n\cdots(b_q)_n}\cdot
\frac{x^n}{n!}$$
is denoted
$${_p}F_q=(a_1,\ldots,a_p;b_1,\ldots,b_q;x),$$
with the $x$ omitted if $x=1$.
The {\em Chu-Vandermonde} identity is
$${_2}F_1(-n,-b;c)=\frac{(c+b)_n}{(c)_n}.$$
The proof of \eqref{hypergeom} is as follows:
\begin{eqnarray*}
\sum_{m=0}^p\frac{(-1)^m(4p-2k-2m)!}{m!(2p-2m)!(2p-k-m)!}&=&
\sum_{m=0}^p\frac{(-1)^m(-1)^{2m}(4p-2k)!/(-4p+2k)_{2m}}{m!((-1)^{2m}
(2p)!/(-2p)_{2m})((-1)^m(2p-k)!/(-2p+k)_m)}\\
&=&\frac{(4p-2k)!}{(2p)!(2p-k)!}\sum_{m=0}^p\frac{(-2p)_{2m}(-2p+k)_m}
{m!(-4p+2k)_{2m}}\\&=&\frac{(4p-2k)!}{(2p)!(2p-k)!}\sum_{m=0}^p
\frac{2^{2m}(-p)_m\left(\frac{-2p+1}{2}\right)_m(-2p+k)_m}{m!2^{2m}(-2p+
k)_m\left(\frac{-4p+2k+1}{2}\right)_m}\\&=&\frac{(4p-2k)!}{(2p)!(2p-k)!}
\sum_{m=0}^p
\frac{(-p)_m\left(\frac{-2p+1}{2}\right)_m}{m!\left(\frac{-4p+2k+1}
{2}\right)_m}\\&=&\frac{(4p-2k)!}{(2p)!(2p-k)!}
\sum_{m=0}^\infty\frac{(-p)_m\left(\frac{-2p+1}{2}\right)_m}{m!\left(
\frac{-4p+2k+1}{2}\right)_m}\\&=&\frac{(4p-2k)!}{(2p)!(2p-k)!}{_2}F_1
\left(-p,\frac{-2p+1}{2};\frac{-4p+2k+1}{2}\right)\\&=&\frac{(4p-2k)!}
{(2p)!(2p-k)!}\cdot\frac{(k-p)_p}{\left(\frac{-4p+2k+1}{2}\right)_p}\\
&=&\frac{(4p-2k)!}{(2p)!(2p-k)!}\cdot\frac{(k-p)\cdots(k-1)(-2)^p}{
(4p-2k-1)(4p-2k-3)\cdots(2p-2k+1)}.
\end{eqnarray*}
Now, for $k=1,\ldots,p$, this last expression clearly vanishes.  For $k=0,$
the last line reduces to
\begin{eqnarray*}
\frac{(4p)!\,(p)!\,2^p}{(2p)!\,(2p)!\,(4p-1)(4p-3)\cdots(2p+1)}&=&
\frac{(4p)!\,(p)!\,2^p\,(2p-1)!!}{(2p)!\,(2p)!\,(4p-1)!!}\\&=&\frac{2^{3p}\,
(p)!\,(2p-1)!!}{(2p)!}\\&=&2^{2p}.
\end{eqnarray*}
This completes the proof.
\bibliographystyle{plain}
\bibliography{eigenfunctions}

\end{document}